\documentclass[12pt]{amsart}
\usepackage{amsfonts}
\usepackage{amsmath,bm,amsthm,amssymb}
\usepackage{bbm}
\usepackage{bigints}
\usepackage{caption}
\usepackage{color}
\usepackage{enumerate}
\usepackage{enumitem}
\usepackage[T1]{fontenc}
\usepackage{fullpage}
\usepackage[left=2cm,right=2cm,top=2.5cm,bottom=2.5cm]{geometry}
\usepackage{graphicx}
\usepackage{hyperref}
\usepackage{mathrsfs}
\usepackage{multirow}
\usepackage[square,sort,comma,numbers]{natbib}
\usepackage{natbib}
\usepackage{pgfplots}
\usepackage{setspace}
\usepackage{siunitx}
\usepackage{subcaption}
\usepackage{tikz}
\usepackage{times}

\usetikzlibrary{patterns,calc}
\pgfplotsset{compat=1.18}

\newcommand{\EXCLUDE}[1]{}

\newcommand{\remove}[1]{}
\def\P{\mathbb{P}}
\def\E{\mathbb{E}}
\newcommand{\beqn}{\begin{equation}}
	\newcommand{\eeqn}{\end{equation}}
\newcommand{\beq}{\begin{eqnarray}}
	\newcommand{\eeq}{\end{eqnarray}}
\newcommand{\beqq}{\begin{eqnarray*}}
	\newcommand{\eeqq}{\end{eqnarray*}}
\def\:{:\,}

\newtheorem{theorem}{Theorem}[section]

\newtheorem{lemma}[theorem]{Lemma}
\newtheorem{prop}[theorem]{Proposition}

\newtheorem{application}[theorem]{Application}

\def\bet{\begin{theorem}}
	\def\ent{\end{theorem}}
	\def\k{\kappa}

\def\bK{\mathbf{K}}

\def\om{\omega}

\newcommand{\m}{\mu}
\newcommand{\Om}{\Omega}
\newcommand{\Ga}{\Gamma}

\def\bK{\mathbb{K}}

\newcommand{\cS}{{\mathcal S}}

\def\1{\mathbf{1}}

\def\e{\varepsilon}

\def\lab{\label}

\def\f{\frac}

\def\a{\alpha}
\def\su{\subseteq}
\def\s{\sigma}
\def\r{\rho}

\def\ff{\infty}
\def\R{\mathbb R}
\def\K{\mathbb K}
\def\Z{\mathbb Z}

\def\om{\omega}
\def\Om{\Omega}

\def\bfx{\boldsymbol{x}}
\def\bfy{\boldsymbol{y}}
\def\bfz{\boldsymbol{z}}
\def\bfo{\boldsymbol{o}}
\def\es{\varnothing}
\def\de{\delta}
\def\bel{\begin{lemma}}
\def\bep{\begin{proof}}
\def\enp{\end{proof}}
\def\enl{\end{lemma}}

\def\d{{\rm d}}
\def\tff{\uparrow\infty}

\def\ms{\mathsf}
\renewcommand\le{\leqslant}
\renewcommand\ge{\geqslant}
\def\one{\mathbbmss{1}}
\def\wt{\widetilde}
\def\mc{\mathcal}
\def\PP{\mc P}
\def\vp{\varphi}
\def\been{\begin{enumerate}}
\def\enen{\end{enumerate}}
\def\bee{\begin{application}}
\def\ene{\end{application}}
\def\im{\item}
	
	\def\ms{\mathsf}
	\def\co{\colon}

\begin{document}
	
	
\title[Poisson approximation in weighted RCMs]{\large{Poisson approximation of fixed-degree nodes in weighted random connection models}}

\author{Christian Hirsch}
\author{Benedikt Jahnel}
\author{Sanjoy Kumar Jhawar}
\author{P\'eter Juh\'asz}

\address[Christian Hirsch, P\'eter Juh\'asz]{Department of Mathematics, Aarhus University, Ny Munkegade 118, 8000 Aarhus C, Denmark}
\email{hirsch@math.au.dk, peter.juhasz@math.au.dk}
\address[Christian Hirsch]{DIGIT Center, Aarhus University, Finlandsgade 22, 8200 Aarhus N, Denmark}
\address[Benedikt Jahnel, Sanjoy Kumar ~Jhawar]{Weierstrass Institute for Applied Analysis and Stochastics Berlin, Mohrenstrasse 39, 10117 Berlin, Germany}
\email{jahnel@tu-braunschweig.de, jhawar@wias-berlin.de}
\address[Benedikt Jahnel]{
Institut f\"ur Mathematische Stochastik, Technische Universit\"at Braunschweig,
Universit\"atsplatz 2, Braunschweig}

\date{\today}
\noindent  \thanks{  }
\maketitle
\begin{abstract}
We present a process-level Poisson-approximation result for the degree-$k$ vertices in a high-density weighted random connection model with preferential-attachment kernel in the unit volume.
Our main focus lies on the impact of the left tails of the weight distribution for which we establish general criteria based on their small-weight quantiles.
To illustrate that our conditions are broadly applicable, we verify them for weight distributions with polynomial and stretched exponential left tails.
The proofs rest on truncation arguments and  a recently established quantitative Poisson approximation result for functionals of Poisson point processes.
\end{abstract}


\section{Introduction}

Spatial random networks are found in a wide variety of applications ranging from social networks over materials science to telecommunication systems~\cite{haenggi2012stochastic, jahnel2020probabilistic}.
In particular, in the context of such networks, it is essential to estimate the probability that we observe extreme realizations of the key network characteristics, and to understand the reasons leading to such extreme behavior.
This need motivates the extension of the classical findings from extreme-value theory to the context of spatial random networks.

A seminal paper in this context is~\cite{penrose} which studies the asymptotic behavior of the number of degree-$k$ nodes in the inhomogeneous random connection model (RCM).
More precisely, the main result of~\cite{penrose} states that the number of degree-$k$ nodes converges to a Poisson distribution under a suitable scaling of the connectivity threshold in the connection function.

While the extensions of Stein's method developed in~\cite{penrose} are highly interesting from a mathematical point of view, it is not always easily applicable in practice.
The reason is that the RCM can only produce light-tailed degree distributions, whereas many real-world networks exhibit heavy tails.
In order to overcome this limitation,~\cite{ejp} extended the results to the scale-free RCM introduced in~\cite{Deprez2018,sfp}.
These networks produce heavy-tailed degree distributions by endowing the vertices with suitable weights that strongly influence their ability to connect to other vertices.

While the scale-free RCM comes with many parameters, it is an important finding in~\cite{ejp} that most of them influence the connectivity threshold only through multiplication by a constant.
This could point to a potential weakness of the scale-free RCM since, depending on the application, we would expect also a wide variety of extreme-value scalings.

In this paper, we resolve this potential misconception by showing that, indeed, scale-free RCMs can give rise to a wide variety of different extreme-value scalings.
We stress that these findings do not contradict the results in~\cite{ejp}, since we also allow for the variation of the left tail of the weight distribution.
Due to its importance for the degree distribution, the majority of the existing literature focuses exclusively on studying the effects of the right tail.
However, one of the core findings of our work is that, for the extreme-value behavior of the number of degree-$k$ nodes, it is the left tail of the weight distribution that is of crucial importance.

While this may be surprising at first sight, it is not entirely unexpected.
Indeed, our analysis shows that the most likely reason for seeing a constant order of isolated nodes is that these nodes have an extremely small weight, making it easier for these nodes to be isolated.
On the other hand, it is the left tail that determines how difficult it is for a node to have extremely small weight.
Hence, by allowing modification to the left tail of the distribution of the weights, we can observe a variety of different extreme-value behaviors.

Let us illustrate these findings through simulations set up with a constant order of isolated nodes.
In the left panel of Figure~\ref{fig:fig0}, we consider a typical realization of an {\em isolated node} (red) in a weighted RCM of intensity 1000, whose weight distribution has a power-law left tail of tail index 2.
The horizontal axis corresponds to the positions of the nodes, whereas the vertical axis shows their weights.
Edges are not shown.
Loosely speaking, we see that most of the network looks like a typical realization of a Poisson point process.
However, the weight of the red node at the origin is atypically small, making it easy to avoid connections.
Our main result makes this intuition precise by providing a quantitative prediction for how small the weight of the origin must be in order to see a constant order of isolated nodes.
In the right panel of Figure~\ref{fig:fig0}, we show a log-log plot of the weight of a typical isolated node against the point-process intensity.
The plot is approximately linear with slope $-0.4682$, and thus close to our theoretical prediction $-1/2$.
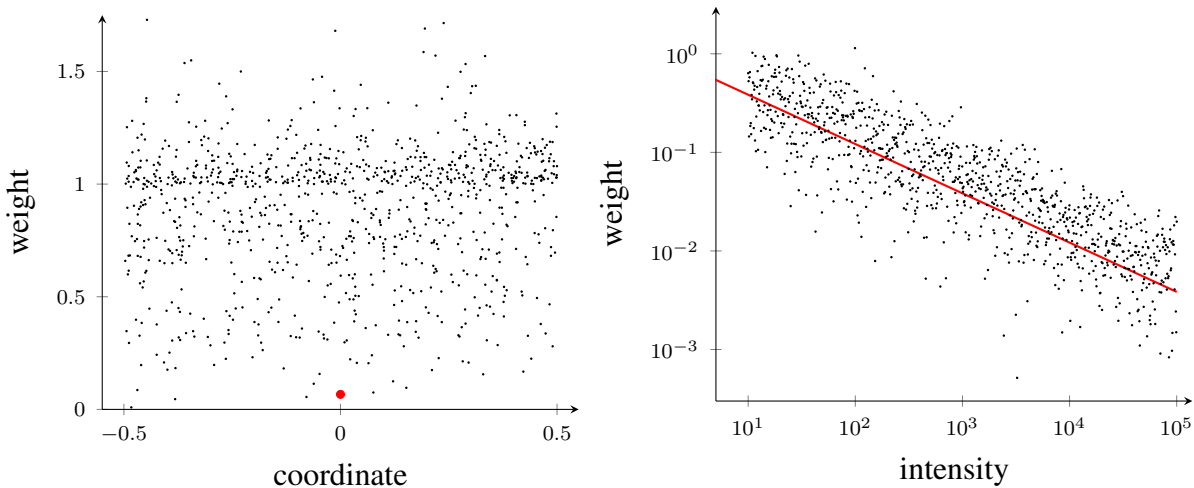
\begin{figure}[htb] \centering
    \begin{tikzpicture}
      \begin{axis}[
        width=0.45\textwidth,
        axis x line=bottom,
        axis y line=left,
        xmin=-0.55, xmax=0.55,
        ymin=0, ymax=1.75,
        xtick={-0.5, 0, 0.5},
        xlabel=coordinate,
        ylabel=weight,
        minor tick style={draw=none},
        ticklabel style={font=\tiny}
      ]
        \addplot[only marks, mark size=0.25pt] table [col sep=comma, x=coordinate, y=mark] {point_coordinates_marks.csv};
        \addplot[mark=*, mark size=1.5pt, only marks, red] coordinates {(0, 0.0667)};
      \end{axis}
    \end{tikzpicture}
    \begin{tikzpicture}
      \begin{axis}[
        width=0.45\textwidth,
        axis x line=bottom,
        axis y line=left,
        xmin=5, xmax=140000,
        ymin=3e-4, ymax=3,
        xmode=log, ymode=log,
        xlabel=intensity,
        ylabel=weight,
        minor tick style={draw=none},
        ticklabel style={font=\tiny}
      ]
        \addplot[only marks, mark size=0.25pt] table [col sep=comma, x=intensity, y=typical_mark] {intensity_typical_mark.csv};
        \addplot[domain=0.1:1e5, red, thick]{1.2144 * x^(-0.5)};
      \end{axis}
    \end{tikzpicture}
  \caption{
    Realization of an isolated node (red) in a weighted RCM with intensity 1,000 and weight distribution with power-law left tails of tail index 2 (left).
    Log-log plot of the weight of a typical isolated point against varying point-process intensities (right).
  }
  \label{fig:fig0}
\end{figure}

The main contribution of our paper can be summarized as follows.
\been
\im While the results in~\cite{penrose,ejp} tentatively indicate that the rare-event behavior of the degree-$k$ nodes is not strongly affected by the parameters, we show that, in fact, a wide variety of extreme-value scalings can be obtained.
This is achieved by modifying the left tail of the weight distribution.
As specific examples, we consider the case where the left tail is of power-law or of Fr\'echet type.
\im Our results go beyond the scale-free RCM and consider the \emph{weighted RCM (WRCM)} introduced in~\cite{glm,glm2,gracar2022recurrence,komjathy2020explosion}.
WRCMs specify a kernel function and therefore allow to consider models of spatial preferential attachment.
\im In contrast to~\cite{penrose,ejp}, we not only look at the number of degree-$k$ nodes, but also on their spatial distribution.
In other words, we prove convergence of the degree-$k$ nodes to a Poisson point process.
This is accomplished through the application of a recently developed functional Poisson-approximation result from~\cite{BSY202}.
\enen

The rest of the manuscript is organized as follows.
In Section~\ref{sec:mod}, we recall the definition of the WRCM and state Theorem~\ref{thm:apx} as our main result on the Poisson approximation of the degree-$k$ nodes.
Loosely speaking, the precise rare-event behavior is encoded in a characterizing equation that prominently involves the weight distribution.
Next, in Section~\ref{sec:exa}, we illustrate that the conditions on the weight distribution are meaningful as they cover a wide range of natural models.
Finally, in Section~\ref{sec:proofs}, we present the proofs of the above results.

%
%
\section{The inhomogeneous random connection model and main results}
\label{sec:mod}
We now recall from~\cite{glm,glm2} the precise definition of the kernel-based spatial random networks that are the object of our study.
We denote by $S=[0, 1]^d$ the $d$-dimensional unit cube with $d\ge 1$.
We consider henceforth a random graph with vertex set given by a homogeneous Poisson point process $\PP_s$ on $\R^d$ with intensity $s > 0$.
Independently to each $x \in \PP_s$, we associate a random weight $W_x$ drawn from a distribution with cumulative distribution function $F(w) := F_W(w) = \P(W \leq w)$ on $(0, \ff)$.
The probability that there is an edge between any two vertices $x, y \in \PP_s$ is a function of their distance and their weights $W_x$ and $W_y$, i.e., 
\begin{equation} p_s(x, W_x; y, W_y) := \vp\big(|B_{|x - y|}(o)|/(v_s\k(W_x,W_y))\big), \lab{eq:connection_function_s} \end{equation}
where $v_s$ is the scaling factor depending on the intensity, $|B_{|x - y|}(o)|$ is the volume of the centered Euclidean ball with radius $|x - y|$, and $\k$ and $\vp$ are the \emph{kernel} and the \emph{profile function} of the model that are specified as follows.
As kernel, we consider the preferential-attachment kernel from~\cite{glm,glm2}, i.e.,
$$ \k(w_1,w_2) = (w_1\wedge w_2) (w_1\vee w_2)^a $$ 
for some $a\ge 0$.
The profile function is a non-negative function $\vp \co [0, \ff) \to [0, 1]$ satisfying the normalization condition $\int_0^\ff\vp(r) \d r = 1$, and regularly varying at infinity with tail index $\a > 1$.
We assume that $1-F$ is also regularly varying at infinity with tail index $\beta>a\a$, in particular, $\mu_{a\a}<\ff$ with $\mu_r:=\E[W^{r}]$.
The resulting random graph is denoted by $G(\PP_s, v_s)$.
We highlight that the parameter $sv_s$ has a natural interpretation in terms of the network model.
Indeed, as we will see in Application~\ref{ee:dego} in Section~\ref{sec:proofs}, $sv_s$ is the order of the expected number of neighbors of a typical network node.

We are interested in the spatial distribution of nodes with a given degree, i.e., 
\begin{align} \label{def:degPP} \xi_{s} := \xi_{s,k} := \sum_{x \in \PP_s \cap [0,1]^d} \one \{ \deg(x) = k \text{ in } G(\PP_s, v_s) \} \delta_x. \end{align}
Note that $\deg(x)$ is the number of all points connected to $x$, including those that lie outside of the unit cube $[0,1]^d$.
In order to specify the scaling $v_s$, we first consider $D_s := D_{s,k} := \xi_{s,k} ([0,1]^d)$, the number of degree-$k$ vertices in $G(\PP_s, v_s)$ that are contained in $[0, 1]^d$.
We identify the correct scaling for $v_{s,k}$ such that $\E[D_{s,k}]$ is constant.
For this, let us introduce the decomposition
\begin{align} \label{notation} h(w) := w \mu_+(w) + w^a \mu_-(w) := w \E \big[ \one \{W \ge w\} W^a \big] + w^a \E \big[ \one \{W < w\} W \big], \end{align}
of $h(w) = \E[\kappa(w,\cdot)]$ and consider the scaling $v_s$ defined as the largest solution of the equation
\begin{equation} k! =s\E[(s v_{s,k} h(W))^k\exp(-s v_{s,k} h(W))], \tag{\bf SCG} \lab{eq:scaling} \end{equation}
for $k \ge 0$.
We note that such a solution must exist at least for all sufficiently large $s$.
Indeed, take for example $v_{s,k} = c/s$, then Equation~\eqref{eq:scaling} can be rewritten as $s = k! / \E[(ch(W))^k \exp(-ch(W))]$, where the right-hand side tends to infinity for $c\to\ff$.
The following result establishes the correct scaling, with its proof presented in Section~\ref{sec:proofs}.
\begin{lemma}[Expected typical degree]\lab{theorem:Ex_poisson_conv}
Let us fix $k \geq 0$ and consider the random graph $G(\PP_s, v_{s,k})$ with the connection function $p_s$ of the form~\eqref{eq:connection_function_s} and scaling parameter $v_{s,k}$ as defined by~\eqref{eq:scaling}.
Then,
\begin{equation} \E[D_{s,k}]=1. \lab{eq:exp_d_k} \end{equation}
\end{lemma}

While the definition of $v_{s,k}$ is indirect, in Section~\ref{sec:exa}, we illustrate that the order of $v_{s,k}$ can be computed for given natural choices of the weight distribution as a function of the intensity $s$.
The value of the parameter $v_{s,k}$ has important implications for the network topology as it reveals the order of the expected number of neighbors of a typical node, $s v_{s,k}$.
More precisely, we will see that, for polynomial tails, $s v_{s,k}$ is of polynomial order, while for stretched exponential tails, it is of polylogarithmic order.
This reflects the intuition that for polynomial tails, it is more likely that low-weight nodes appear, which means that even for relatively large values of $s v_{s,k}$, it is reasonably likely that a low-weight node is isolated.
For stretched exponential tails, it is less likely to create low-weight nodes, which means that even for smaller values of $s v_{s,k}$, it is unlikely that a low-weight node is isolated.
These observations are to be contrasted with the finding from~\cite{ejp} that for a lower-bounded weight distribution, $s v_{s,k}$ is of a much smaller, namely, logarithmic order.

\medskip
Having established the convergence of the expected degree counts in Lemma~\ref{theorem:Ex_poisson_conv}, in Theorem~\ref{thm:apx} below, we prove the convergence of the degree distribution itself in the sense of a Poisson point process approximation result.
Note that this result is only valid under certain assumptions on the distributions of the weights $W$, which we now collect.
These assumptions are rather technical but are a key component of the approximation arguments in our proof.
In Section~\ref{sec:exa}, we illustrate how to verify these assumptions and present examples for weight distributions exhibiting a variety of left tails.

A central role in our proof is played by the $1/(2s)$-quantile of the weight distribution, which we denote by $w_s$.
That is, $F(w_s) = 1/(2s)$.
The importance of this quantity stems from the intuition that it is a first indication for the typical weight of an isolated node in $[0, 1]^d$.
Indeed, nodes of weights much smaller than $w_s$ are unlikely to appear in $[0, 1]^d$, whereas nodes of weights much larger than $w_s$ are unlikely to be isolated.
We stress that the precise interpretation of \emph{much smaller} and \emph{much larger} may depend on the tail distribution.
This is one of the main reasons why the following assumptions are rather technical.
Let $\de:=(\a - 1)/2$.
Our assumptions require that, for some $K > 0$ and $\eta \in(0,1)$,
\been[label=\textbf{A.\arabic*},ref=A.\arabic*]
\im \label{as:1} $\liminf_{s\uparrow\ff} s v_{s,k} w_s^\eta/\log(s)=\ff$,
\im \label{as:2} $\limsup_{s\uparrow\ff} \log(s) w_s^{- (K + 1)(1 - \eta)}F(w_s^\eta)=0$ and 
\im \label{as:3} $\limsup_{s\uparrow\ff} \log(s) w_s^{ (K\de - 1)(1 - \eta)}=0$.
\enen
In order to present our main result, let $\zeta$ denote a Poisson point process with intensity ${\rm{Leb}}(\d x):=\one\{x\in [0,1]^d\}\d x$ and let
$$d_{\rm KR}(\xi,\xi'):=\sup\{|\E[h(\xi)]-\E[h(\xi')]|\colon h\in{\rm LIP}\}$$
denote the Kantorovich--Rubinstein distance between the distributions of the two processes $\xi$ and $\xi'$.
Here ${\rm LIP}$ is the class of measurable $1$-Lipschitz functions with respect to the total variation distance on the space of finite point configurations on $[0,1]^d$.
\begin{theorem}[Poisson approximation]\label{thm:apx}
Let us fix $k \geq 0$ and consider the random graph $G(\PP_s, v_{s,k})$ with the connection function $p_s$ of the form~\eqref{eq:connection_function_s} and 
the scaling parameter $v_{s,k}$ as defined by~\eqref{eq:scaling}.
Then,
\begin{equation}\label{eq:poisson_D_k_b}
\lim_{s\tff} d_{\rm KR}(\xi_{s,k}, \zeta) = 0.
\end{equation}
\end{theorem}

We note that Theorem~\ref{thm:apx} is the analog of~\cite[Theorem 3.2]{ejp} for the case where the weight distribution has positive mass arbitrarily close to 0.
More precisely, we note that in~\cite{ejp}, the connection probability is given by 
$$ 1 - \exp \big( - \eta r_s^\a W_x W_y / |x-y|^\a \big) $$
for some parameters $\eta$, $r_s$, $\a$.
Hence, this can be written as $\vp\big( |B_{|x - y|}(o)| / (v_{s,k} \k(W_x', W_y')) \big)$, where $v_{s,k} := r_s^{d}$, $W_x' := W_x^{d/\a}$, the parameter $a = 1$ in the kernel $\k$ and
$$ \vp(r) := 1 - \exp\big( -\eta |B_1(o)|^{\a/d} r^{-\a/d} \big). $$

We also note that in~\cite{penrose}, a straightforward extension of the arguments for the degree-$k$ vertices also yields the Poisson approximation for size-$k$ components.
However, in our setting, the introduction of the weights makes the analysis of the size-$k$ components substantially more involved.
Indeed, such an analysis would rely on a highly delicate configurational analysis of the weights in such components.
These weights need to be small enough to ensure that there are no connection to outside nodes, while simultaneously they need to be large enough to ensure the connectivity between the nodes in the component.
While such an analysis is not entirely out of range, it would require additional constraints on the weight distribution as well as a substantially more refined analysis.
Hence, to give a focused presentation of the main ideas, we refrain from carrying out such an analysis here.

Moreover, as mentioned in the introduction, the main tool of the proof is a recently developed functional Poisson-approximation result from~\cite[Theorem 4.1]{BSY202}.
Here, we note that~\cite{BSY202} also gives a functional Poisson-approximation result for the nearest-neighbor radii.
However, this result heavily relies on the specific form of the isolation probability for the standard random geometric graph on a Poisson point process.
In particular, such a result does not extend easily to the present setting, where the isolation probability depends in a complicated way both on the weights and on the profile function.

Finally, we note that, while~\cite[Theorem 4.1]{BSY202} provides a rate of convergence, we refrain from stating such rates here.
This is because the complexity of our model forces us to make approximations at several instances that are presumably suboptimal.
Hence, while it would be possible to extract specific convergence rates from our proof, they would be far from optimal as well.
Since a streamlined proof without tracking the rates is substantially more accessible, we decided to present the proof in this form.

%
%
\section{Examples}
\label{sec:exa}

The goal of this section is to provide examples for the weight distributions and show that they fulfill the Assumptions~\ref{as:1},~\ref{as:2},~\ref{as:3} listed in Section~\ref{sec:mod}.
In Sections~\ref{ex:pol} and~\ref{ex:exp}, we discuss examples for weight distributions with polynomial and stretched exponential left tails, respectively.

Let us start by stating some apriori estimates of our parameters.
Let $\s_s := \s_{s,k} := s v_{s,k}$ and recall the definitions in~\eqref{notation}.
%
%
\begin{lemma}\label{lem:estimates}
It holds that
\begin{enumerate}
    \item $\mu_-(w)\le wF(w)$,
    \item $w^{a}(1-F(w)) \le \mu_+(w)\le \mu_a$,
	    \im  $\lim_{w \to 0} \m_+(w) = \mu_a$,
		\im $\lim_{w \to 0} h(w) / w = \mu_a$, and
    \item $\s_{s,k} w_s \in O(\log s)$.
\end{enumerate}
\end{lemma}
\begin{proof}
The first three statements are immediate.
For the fourth statement, as $w \downarrow 0$, we can bound the fraction of the two terms of $h(w)$ as
\[ \lim \limits_{w \downarrow 0} \f{w^a \m_- \left( w \right)}{w \m_+ (w)} \le \lim \limits_{w \downarrow 0} w^{a} \, \f{F(w)}{\mu_a} = 0. \]
For the fifth statement, let $c' > 0$ be such that $\mu_-(c') > 0$.
Then, introducing indicators for the events $W < w_s$ and $w_s \le W$ we can bound for some constant $c > 0$,
\begin{align*}
1 = s \E \Big[ \f1{k!} (\s_s h(W))^k\exp(- \s_s h(W)) \Big] &\le 1/2 + s \E \Big[ \f1{k!}(\s_s h(W))^k \exp(- \s_s h(W)) \one \{ W > w_s \} \Big] \\
        &\le 1/2 + cs \E \big[ \exp(-\s_s h(W) / 2) \one \{ w_s < W < c \} \big] \\
        &\le 1/2 + cs \exp(- \s_s w_s \mu_-(c')),
\end{align*}
where we also used that $\sup \{ x^k \exp(-x/2) \colon x \ge 0\} < \ff$.
Hence, $\s_s w_s \mu_-(c') \le 2 \log(2cs)$, and thus the result follows.
\end{proof}

The starting point of the computations in this section is the Equation~\eqref{eq:scaling}.
We use that $\lim_{s\tff} \s_s =\ff$.
This is so because the expectation in~\eqref{eq:scaling} must be zero in the limit so that the left-hand side of~\eqref{eq:scaling} is constant.

%
%
\subsection{Polynomial left tails}
\label{ex:pol}
First, we consider the setting with polynomial left tails, where we assume that 
\[ F(w) = pw^\r \quad \text{ for }p, \r > 0 \text{ and for all }w \le b\text{ for some }b > 0.\]
Note that the $1/(2s)$-quantile $w_s$ of the weight distribution is given by
\[ \P(W \le w_s) = 1/(2s) \quad\Longleftrightarrow\quad p \, w_s^\r = 1/(2s) \quad\Longleftrightarrow\quad w_s = \left( 2 p s \right)^{-1/\r}. \]

%
%
Let us first establish the scaling of $\s_s$.
\begin{lemma}[$\s_s$-scaling for polynomial left tails]
Let $k \geq 0$ and let $v_{s}$ be as defined in~\eqref{eq:scaling}.
Then
\[ \lim \limits_{s \tff} \f{\s_{s}^\r}s = \f{p\r}{k!\mu_a^\r} \, \Ga \left( k + \r \right). \]
\label{lemma:equal_limit_L}
\end{lemma}
%
%
\bep
Let $0 < \e < b$.
Then, using~\eqref{eq:scaling}, we have that
\[ \begin{aligned}
\lim \limits_{s \tff} \f{\s_s^\r}s &= \lim \limits_{s \tff} \f{\s_s^\r}{k!} \, \E \big[ \big( \s_s h(W) \big)^k \exp \big(- \s_s h(W) \big) \big] \\
&= \lim \limits_{s \tff} \f{\s_s^\r}{k!} \big( \E \big[ \big( \s_s h(W) \big)^k \exp \big(- \s_s h(W) \big) \, \one \big\{ W \le \e \big\} \big] \big. \\
&\hspace{1.5cm} + \big. \E \big[ \big( \s_s h(W) \big)^k \exp \big(- \s_s h(W) \big) \, \one \big\{ W > \e \big\} \big] \big),
\end{aligned} \]
where the second term decays exponentially fast in $\s_s$ since $h(W) > 0$ whenever $W > \e$.
Thus, we keep only the case when $W \le \e$, which leads to
\[ \lim \limits_{s \tff} \f{\s_s^\r}s = \lim \limits_{s \tff} \f{\s_s^\r}{k!} \E \big[ \big( \s_s h(W) \big)^k \, \exp \big(- \s_s h(W) \big) \, \one \big\{ W \le \e \big\} \big]. \]
Note that, for $w\le b$, the weight distribution has the density $f(w) = p\r w^{\r - 1}$.
Since $\e \downarrow 0$, we can use that $\lim_{w \downarrow 0} h(w) / w = \mu_a$.
We use Lemma~\ref{lem:estimates} to see that 
\[ \lim \limits_{s \tff} \f{\s_s^\r}s = \lim_{\e \downarrow 0} \lim \limits_{s \tff} \f{\s_s^{\r+k} p \r\mu_a^k}{k!} \int_0^{\e} w^{k + \r - 1} \exp \big( - \s_s \mu_a \, w(1+o_\e(1)) \big) \, \d w. \]
Finally, substituting $u := \s_s \mu_a w$ we have that
\[ \begin{aligned} \lim \limits_{s \tff} \f{\s_s^\r}s &= \lim_{\e \downarrow 0} \lim \limits_{s \tff} \f{p\r}{k!\mu_a^\r} \int_0^{\s_s\e\mu_a} u^{k + \r - 1} \exp \big( - u(1 + o_\e(1) \big) \, \d u 
= \f{p \, \r}{k! \mu_a^\r} \Ga \big( k + \r \big),
\end{aligned} \]
as asserted.
\enp

After having shown  Lemma~\ref{lemma:equal_limit_L}, we verify Assumptions~\ref{as:1}--~\ref{as:3}.
First, we deal with Assumption~\ref{as:1}.
As $w_s = (2 p s)^{-1/\r}$, we have that
\[ \begin{aligned}
 \liminf_{s \tff} \f{\s_s(2 p s)^{-\eta/\r}}{\log(s)} &= \liminf_{s \tff} \f{(2 p s)^{-\eta/\r}}{\log(s)} \left( \f{s p\r}{k!\mu_a^\r}\Ga \left( k + \r \right) \right)^{1/\r}\\ 
&=\liminf_{s\tff}\left( \f{p^{1-\eta} \r \Ga \left( k + \r \right)}{2^{\eta}k!\mu_a^\r} \right)^{1/\r} \f{s^{(1-\eta)/\r}}{\log(s)} = \ff,
\end{aligned} \]
where in the second step we used Lemma~\ref{lemma:equal_limit_L}.

Now, we choose $K := 2/ \delta$ and subsequently $\eta$ sufficiently close to $1$ such that $(K+1)(1-\eta)/\rho < \eta$, or, equivalently,
\[ \f{(1 + K)/\r}{1 + (1 + K)/\r} < \eta < 1. \]
We now turn our attention to Assumption~\ref{as:2}.
Note that $ F(w_s^\eta) = p \, w_s^{\eta \, \r} =p/(2 p s)^{\eta} $.
Therefore, 
\[ \begin{aligned}
\limsup_{s \tff} \log(s) \, w_s^{- \left( K + 1 \right) \left( 1 - \eta \right)} \, F(w_s^\eta) &= \limsup_{s \tff} \log(s) \, (2 p s)^{{\left( K + 1 \right) \left( 1 - \eta \right)}/{\r }} \, \f{p}{(2 p s)^{\eta}} \\
&= p \, \limsup_{s \tff} \log(s) \, (2 p s)^{{\left( K + 1 \right) \left( 1 - \eta \right)}/\r - \eta} = 0,
\end{aligned} \]
where we used that the exponent is negative for the chosen $\eta$.
Thus, Assumption~\ref{as:2} is satisfied.

Finally, for Assumption~\ref{as:3} we have that
\[ \limsup_{s \tff} \log(s) \, w_s^{(K \delta - 1)(1 - \eta)} = \limsup_{s \tff} \log(s) (2 s p)^{- {(K \delta - 1)(1 - \eta)}/\r} = 0, \]
as the exponent is negative by the choice of $K$.

%
%
\subsection{Stretched exponential left tails}
\label{ex:exp}

Now, we consider stretched-exponential left tails.
That is,
\[ F(w) := p\exp(-w^{-\r}) \quad \text{ for } p, \r > 0 \text{ and for all }w \le b\text{ for some }b > 0. \]
Note that the $1/(2s)$-quantile $w_s$ of the weight distribution is given by
\[ \P(W \le w_s) = 1/(2s) \quad\Longleftrightarrow\quad p\exp \left( - w_s^{-\r} \right) =1/(2s) \quad\Longleftrightarrow\quad w_s = \log \left( 2 p s \right)^{-1/\r}. \]
The main work of this subsection is to identify the $\s_s$-scaling.
\begin{lemma}[$\s_s$-scaling for stretched exponential left tails]
Let $v_s$ be as defined in~\eqref{eq:scaling}.
Then for $k \ge 0$,
\begin{equation}
      0 < \liminf_{s \tff}  \f{\s_s}{\log(s)^{1+1/\r}}  \le \limsup \limits_{s \tff}  \f{\s_s}{\log(s)^{1+1/\r}}  < \ff.
      \lab{eq:exp_example_limit}
\end{equation} \label{lemma:exp_example_limit}
\end{lemma}
Since the proof of Lemma~\ref{lemma:exp_example_limit} is rather technical, we first explain how to verify Assumptions~\ref{as:1}--~\ref{as:3}.
For this, let $\eta = 1/2$ and fix $K > (2 \rho + 1) / \delta$.

First, we examine Assumption~\ref{as:1}.
As $w_s = \log (2 p s)^{- 1/\r}$, using Lemma~\ref{lemma:exp_example_limit}, for some $M > 0$,
\[ \begin{aligned}
\liminf_{s \tff} \f{sv_sw_s^{\eta}}{\log(s)} &= \liminf_{s \tff} \f{\s_s\log (2 p s)^{-\eta/\r}}{\log(s)} \ge \lim_{s \tff} \f1M \, \log (s)^{-\eta/\r - 1} \, \log (s)^{1 +1/\r} = \ff
\end{aligned} \]
where we applied that $\lim_{s\tff} \log(s) / \log(2ps) = 1$.

Let us consider Assumption~\ref{as:2}.
As
\[ F(w_s^\eta) = p\exp \left( - w_s^{- \r \, \eta} \right) = p\exp \big( - \log \left( 2 p s \right)^{1/2} \big), \]
we have that
\[ \begin{aligned}
\limsup_{s \tff} \log(s)w_s^{- \left( K + 1 \right) \left( 1 - \eta \right)}F(w_s^\eta) &= \lim_{s \tff} \log(s) \log \left( 2 p s \right)^{(K + 1)/(2 \r)} p \exp \big( - \log \left( 2 p s \right)^{1/2} \big) \\
&= p \lim_{s \tff} \log(s)^{1 + (K + 1)/(2 \r)} \, \exp \left( - \log(s)^{1/2} \right),
\end{aligned} \]
Noting that $ \log ( \log (s)^{1 + (K + 1) / 2\r} ) \in o( \log(s)^{1 / 2} )$, independently of the value of $K$, we conclude that
\[ \limsup_{s \tff} \log(s) w_s^{- \left( K + 1 \right) \left( 1 - \eta \right)}F(w_s^\eta) = p \limsup_{s \tff} \exp \big( - \log \left( s \right)^{1/2} \big) = 0. \]
Thus, the Assumption~\ref{as:2} holds.

Similarly, we can see that Assumption~\ref{as:3} is satisfied since for our choice of $K$,
\[ \limsup_{s \tff} \log(s) w_s^{\left( K \delta - 1 \right) \left( 1 - \eta \right)} = \lim_{s \tff} \log(s) \log \left( 2 p s \right)^{-(K \delta - 1)/(2\r)} = 0.\]

Now, we prove Lemma~\ref{lemma:exp_example_limit}.
In the proof, we use the Landau notation $f \asymp g$  for $f \in O(g)$ and $g \in O(f)$.

\bep[Proof of Lemma~\ref{lemma:exp_example_limit}]

To prove~\eqref{eq:exp_example_limit} for general $k \ge 0$, we show a lower and upper bound of the value of $\s_s$ from~\eqref{eq:scaling}.
Let us introduce the following notations,
\[ \s_s^- := M^{-1}\log (s)^{1+1/\r} \qquad\text{ and }\qquad \s_s^+ := M\log (s)^{1+1/\r}, \]
for a suitable $M > 0$ chosen below.

Our goal is to show that for some $M > 0$ and all sufficiently large $s$, the largest solution $\s_s$ of~\eqref{eq:scaling} lies in 
$ \s_s^- < \s_s < \s_s^+$.
For a fixed intensity $s$, Figure~\ref{fig:proof_exponential_example} shows the right-hand side
\[ L(s, \s) := s \, \E [ (\s \, h(W))^k \, \exp ( - \s \, h(W))] \]
of~\eqref{eq:scaling} as a function of $\s$ compared to its left hand side $k!$.
\begin{figure}[htb]
  \begin{tikzpicture}
    \begin{axis}[
      xmin=0, xmax=3.25,
      ymin=0, ymax=3.,
      scale only axis=true,
      width=0.75\textwidth,
      height=0.3\textwidth,
      axis lines=middle,
      ytick={0, 1},
      yticklabels={$0$, $k!$},
      xtick={0.001, 1, 1.5, 2},
      xticklabels={$0$, $\s_s^-$, $\s_s$, $\s_s^+$},
      xlabel=$\s_s$,
      ylabel={$L(s, \s)$},
    ]
    \addplot[black, dashed, domain=0:3, samples=2] {1};
    \addplot[black, dashed, samples=2] coordinates {(1, 0) (1, 1)};
    \addplot[black, dashed, samples=2] coordinates {(1.5, 0) (1.5, 1)};
    \addplot[black, dashed, samples=2] coordinates {(2, 0) (2, 2.5)};
    \addplot[black, ultra thick, samples=2] coordinates {(1, 1.) (1, 2.5)};
    \addplot[black, domain=0:3.0, samples=100]{13.5*x*exp(-2*x)};

    \draw[fill=gray!50, opacity=0.5, draw=none] plot[smooth, samples=100] (axis cs: 2,0) -- (axis cs: 2,1) -- (axis cs: 3,1) -- (axis cs: 3,0) -- cycle;
    \draw[fill=white] (axis cs: 1.0, 1) circle (3pt);
    \draw[fill=black] (axis cs: 1.5, 1) circle (2pt);

    \node at (axis cs: 0.65, 1.4) {step 1};
    \draw[black, line width=0.1mm] (axis cs: 0.8, 1.4) -- (axis cs: 1, 1.4);
    \node at (axis cs: 2.5, 1.5) {step 2};
    \draw[black, line width=0.1mm] (axis cs: 2.5, 1.25) -- (axis cs: 2.5, 0.5);
    \node at (axis cs: 1.5, 2.) {goal};
    \draw[black, line width=0.1mm] (axis cs: 1.5, 1.75) -- (axis cs: 1.5, 1.0);

    \end{axis}
  \end{tikzpicture} \caption{
    The right-hand side $L(s, \s)$ of~\eqref{eq:scaling} as a function of $\s$ in three intervals.
    In the first step of the proof, we show that $L(s, \s_s^-) > k!$ for large $s$.
    The second step proves that if $s$ is large and $\s \ge \s_s^+$, then $L(s, \s) < k!$.
    Finally, we use the intermediate-value theorem to show that the largest solution of~\eqref{eq:scaling} must lie in the third interval.
  } \label{fig:proof_exponential_example}
\end{figure}
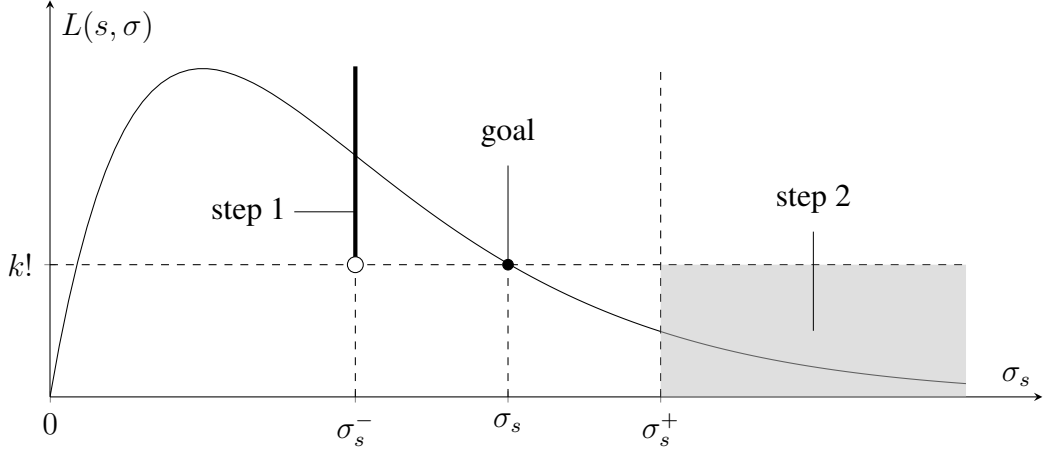

To prove the lemma, we use two steps.
\begin{enumerate}
\item First, we show that $\lim_{s\tff} L(s, \s_s^-) = \ff.$
\item Next, we will see that $\lim_{s\tff} \sup_{\s \ge \s_s^+} L(s, \s) = 0.$
\end{enumerate}
These two steps conclude the proof.
Indeed,  we note that $\lim_{s \tff} L(s, \s_s^-) > k!$ (Step~1), while also $\lim_{s\tff} \sup_{\s \ge \s_s^+} L(s, \s) < k!$ (Step~2).
As $L(s, \s)$ is a continuous function of $\s$, the intermediate-value theorem leads to the conclusion that for large $s$, there is at least one solution of Equation~\eqref{eq:scaling} if $\s_s^- < \s < \s_s^+$.
It also follows from Step 2 that we cannot have a solution if $\s \ge \s_+$.
Thus, the largest solution $\s_s$ must exist and lie in $\s_s^- < \s_s < \s_s^+$, thus proving the lemma.

\medskip
\noindent {\bf Step~1.}
Let us first assume that $\s = \s_s^-$.
Then, for every $K > 0$,
\[ L(s, \s_s^-) = s \, \E \big[ \big( \s_s^- \, h(W) \big)^k \, \mathrm{e}^{ - \s_s^- \, h(W)} \big] \ge s \, \E \big[ \one \big\{ \s_s^- \, h(W) \ge 1, \, W \le K \, w_s \big\} \, \mathrm{e}^{ - \s_s^- \, h(W) } \big]. \]
We choose $K$ and $M$ large enough such that $K^{-\r} + 2M^{-1} K \mu_a < 1$.
Since $w_s \downarrow 0$ as $s \tff$, independently of the chosen $K$ in the indicator function, we use again that $\lim_{w \downarrow 0} h(w) / w = \mu_a$ from part~3 of Lemma~\ref{lem:estimates} to see that
\[ \begin{aligned}
	\lim_{s \tff} L(s, \s_s^-) &\ge \lim_{s \tff} s \, \E \big[ \exp \big( -2 \s_s^- \, W \, \mu_a \big) \, \one \big\{ 2/({\s_s^- \mu_a}) \le W \le K \, w_s \big\} \big] \\
&\ge \lim_{s \tff} s \, \exp \big( -2 \s_s^- \, K \, w_s \, \mu_a \big) \, \P \big( 2/(\s_s^- \mu_a) \le W \le K \, w_s \big).
\end{aligned} \]
We can simplify the above exponential by using the specific form of $\s_s^-$ and $w_s$, which leads to
\[ \exp \big( -2 \s_s^- \, K \, w_s \, \mu_a \big) = \exp \big( - 2{K \, \mu_a} \log(s)^{1 + 1/\r} \, \log(2ps)^{-1/\r}/M \big) = s^{-2 M^{-1}{K \, \mu_a}}, \]
where we also used that $\lim_{s \tff} \log(s) / \log(2ps) = 1$.
We can determine the probability for large $s$ as
\[ \begin{aligned}
\P \big( (\s_s^- \mu_a)^{-1} \le W \le K \, w_s \big) &\ge  \big( \P \big( W \le K \, w_s \big) - \P \big( W \le K \, w_s/2 \big) \big) \\
	&\asymp \big( p \, \big( 2 p s \big)^{-K^{-\r}} - p \, \big( 2 p s \big)^{-(K/2)^{-\r}} \big) \asymp s^{-K^{-\r}}.
\end{aligned} \]
Thus, using these results, we deduce that
\[ \lim_{s \tff} L(s, \s_s^-) \ge \lim_{s \tff}  s^{1 - K^{-\r} - 2M^{-1}K\mu_a } = \ff, \]
since our choice of $K$ ensures $ K^{-\r} + 2M^{-1}K\mu_a < 1$.

\medskip
\noindent {\bf Step 2.}
Here, we have that
\[ \begin{aligned}
\limsup_{s \tff} \sup_{\s \ge \s_s^+} L(s, \s) &= \limsup_{s \tff} \sup_{\s \ge \s_s^+} s \, \E \big[ \big( \s \, h(W) \big)^k \, \exp \big( - \s h(W) \big) \big] \\
&\le \limsup_{s \tff} \sup_{\s \ge \s_s^+} s \, \E \big[ C \, \exp \big( -\s  h(W)/2 \big) \big] \\
	&\le C \, \limsup_{s \tff} s \, \E \big[ \exp \big( -  \s_s^{++} \, h(W) \big) \big],
\end{aligned} \]
where $C := (2k)^k \, \exp(-k)$ and where $\s_s^{++} := \s^+_s/2$.
Now, fixing $K > 0$ with $K^{-\r} \ge 2$ we calculate the upper bound
\[ \lim_{s \tff} s\E \Big[ e^{- \s_s^{++} \, h(W) } \Big] \le \liminf_{s \tff} \Big(s  \E \left[ \exp \left( - \s_s^{++} \, h(W) \right) \, \one \left\{ K \, w_s \le W \le b \right\} \right] + s \P \left( W \not \in  [K \, w_s, b] \right) \Big), \]
where the choice of $K$ implies that the second term is negligible.
Moreover, for $w \in [K \, w_s, b]$, we have $h(w) \ge w\mu_+(b)$ and therefore, since $w_s = \log (2 p s)^{- 1/\r}$, choosing $M > 4/(K\mu_+(b))$ gives that 
\[ \E \left[ \exp \left( - \s_s^{++} \, h(W) \right) \, \one \left\{ K \, w_s \le W \le b \right\} \right] \le \exp \left( - \s_s^{++}K \, w_s \, \mu_+(b) \right) \in o(s^{-2}). \]
We obtain that $\lim_{s \tff} s \E [ e^{- \s_s^{++} \, h(W) } ] = 0$, as asserted.
\enp

\section{Proofs}
\label{sec:proofs}

In order to apply the result~\cite[Theorem~4.1]{BSY202}, it is convenient to express the examined scale-free network via iid marked Poisson point processes.
More precisely, we define 
\begin{equation}
\label{e:defn_varphi}
\xi[\om; v] := \sum_{\bfx \in \om} g_k(\bfx, \om; v)\delta_x
\end{equation}
for locally finite $\om$ in the suitable space $\Omega=\R^d\times [0,\ff)\times [0,1]^{\Z_{\ge 0}}$.
Here $g_k(\bfx, \om; v)$ denotes the indicator that the node $x$ is in $[0,1]^d$ and has degree $k$ in the graph $G(\om; v)$.
We consider $\xi_s=\xi[\wt{\PP}_s;v_s]$ where $\wt{\PP}_s$ is a Poisson point process on $\Omega$ with intensity measure $\K_s(\d \bfx)=s{\rm d} x\otimes \mathbb P^W\otimes \text{Leb}^{\otimes\Z_{\ge 0}}_{[0, 1]}$.
Based on the marks, we draw an edge between any two points $\bfx=(x,W_x,T_x)$ and $\bfy=(y,W_y,T_y)$ with $W_x < W_y$ if $T_y^{(i)}< p_s(x,W_x,y,W_y)$, where $i \ge 0$ is chosen such that $x$ is the $i$-th closest point to $y$ within $\wt{ \PP}_s$.
In words, similarly as in~\cite{penrose}, we encode the randomness associated to the existence of an edge (conditioned on the positions $x,y$ and weights $W_x,W_y$ of the endpoints) into an additional iid marking of the points, where the decision is made by the vertex with larger mark.
The measure $\text{Leb}^{\otimes\Z_{\ge 0}}_{[0,1]}$ then guarantees that any edge in the complete graph has an independent choice, using the fact that with probability one no two points have the same distance.

Now, we can start to break down the proof of our main theorem into two key steps, namely, an approximation step and then the Poisson-convergence proof for the approximating process.
%
%
%
%
As mentioned above, to employ~\cite[Theorem 4.1]{BSY202}, we need to control certain bounding terms.
However, due to the long-range correlations in the spatial random network, it is difficult to directly apply this result.
Moreover, they are also not easily expressible in the usual framework of stopping sets from~\cite{BSY202}.
Therefore, we work with suitable truncations in the weight space and the spatial domain.
Note that the degree of $x$ is determined only by edges between $x$ and points with larger marks, and the mark of $x$ is smaller than $c$.
Thus, we consider the truncated point count 
\[ \hat g^c_k(\bfx,\omega; v) := g_k ( \bfx, \om \cap \Omega_{\ge W_x}; v) \one \{W_x < c\}, \]
where $\Omega_{\ge W_x}=\R^d\times[W_x,\ff)\times [0,1]^{\Z_{\ge 0}}$.
To spatially localize the edge count, let us introduce 
\[ \bar g^{V,c}_k(\bfx, \om; v) = \hat g^c_k( \bfx, \om \cap \Omega^V(x); v), \]
where $\Omega^V(x) = B_{(V/\nu_d)^{1/d}}(x) \times [0, \ff) \times [0,1]^{\Z}$.
Here $\nu_d = |B_1(o)|$ and hence, the ball $B_{(V / \nu_d)^{1 / d}} (x)$ with radius $V / \nu_d$ centered at $x$ has volume $V$.
We write $\Omega_{\ge a}^V(x) = \Omega^V(x) \cap \Omega_{\ge a}$
%
%
and  consider the random variable 
\[ \bar\xi_s^{V_o,w} := \bar\xi_s^{V_o,w} [\wt{\PP}_s; v_s] := \sum_{\bfx \in \wt{\PP}_s} \bar g^{v_s W_x V_o, w}_k (\bfx, \wt{\PP}_s; v_s) \delta_x, \]
where, we fix the cut-off $V_o(s) = w_s^{K(\eta-1)}$ for $\eta \in (0, 1)$ and $K>0$ satisfying the Assumptions~\ref{as:1}--\ref{as:3}.
%
%
The proof of Theorem~\ref{thm:apx} is a direct consequence of the following two statements.
\begin{prop}[Truncations are negligible] \label{pr:neg}
We have that $\limsup_{s\tff}d_{\rm KR}(\xi_s,\bar\xi^{V_o(s),w_s^\eta}_s)=0$.
\end{prop}
%
%
\begin{prop}[Poisson approximation] \label{pr:qa}
We have that $\limsup_{s\tff}d_{\rm KR}(\bar\xi^{V_o(s),w_s^\eta}_s,\zeta)=0$.
\end{prop}
Before presenting the proofs of Propositions~\ref{pr:neg} and~\ref{pr:qa} in Sections~\ref{ss:neg} and~\ref{sec:qa}, respectively, let us collect some supporting results that will be used multiple times later.

%
%
\subsection{Supporting results} \label{ss:deg}

A key property of the considered model is that the expected typical degree conditioned on the typical weight $W_o$ can be expressed in closed form.
This is the content of the following auxiliary result from~\cite[Lemma 4.1]{Deprez2018}.
To make our presentation self-contained, we reproduce here the short proof.
We define an $f$-weighted degree of a marked vertex $\bfx$ as
\[ \ms{deg}_f(\bfx) := \sum_{\bfy \in \wt{\PP}_s \colon y\leftrightarrow x}f(|B_{|x-y|}(o)|, W_x, W_y). \]
In particular, by choosing $f$ as a suitable indicator, we can filter only those neighbors of $x$ satisfying a desired property.
To ease notation, we set
\[ f^*(w_o, w) := \int_{0}^\ff f(uv_s\kappa(w_o,w), w_o, w)\vp(u)\d u,\quad w_o, w > 0. \]

%
\bel[Expected typical degree]
\label{lem:deg}It holds that 
\begin{align*}
	\E[\ms{deg}_f(\bfo)|W_o] &= 
sv_sW_o \E\big[\one\{W \ge W_o\}W^a f^*(W_o, W) \mid W_o\big] \\
	&\phantom=+ sv_sW_o^a \E\big[\one\{W \le W_o\}W f^*(W_o, W)\mid W_o\big].
\end{align*}
\enl
Before proving Lemma~\ref{lem:deg}, we discuss how to simplify it for specific choices of $f$.
Recall that
\begin{equation*}
	h(W_o)=W_o\mu_+(W_o)+W_o^a\mu_-(W_o)= W_o\E\big[\one\{W \ge W_o\}W^a \mid W_o \big] + W_o^a\E\big[\one\{W < W_o\}W \mid W_o \big].
\end{equation*}
\bee[Degree of a typical vertex]
\label{ee:dego}
For $f \equiv 1$, we have that $f^*(w_o, w) = 1$ and hence,
\begin{align*}
	\E[\ms{deg}_f(\bfo)\mid W_o] &= 
sv_sW_o \E\big[\one\{W \ge W_o\}W^a \mid W_o \big] + sv_sW_o^a \E\big[\one\{W \le W_o\}W \mid W_o\big]\\
	&=sv_s\E[\k(W_o, W)\mid W_o] = sv_sW_o\mu_+(W_o) + sv_sW_o^a\mu_-(W_o) 
\end{align*}
\ene

\bee[Out-degree of a typical vertex]
For $f(u, w_o, w) = \one \{w \ge w_o \}$, we have that $f^*(w_o, w) = \one \{ w \ge w_o \}$ and hence,
\[ \E[\ms{deg}_f(\bfo)\mid W_o] = s v_s h(W_o) = s v_s W_o \mu_+(W_o). \]
\ene

\bee[Finite-range truncation] \label{eq:dego1}
For $f(u, w_o, w) = \one\{ u \ge V_o v_s w_o, w \ge w_o\}$, we have that $f^*(w_o, w) = \one\{w\ge w_o\}\int_{V_o/W^a}\vp(u)\d u$ and hence,
\[ \E[\ms{deg}_f(\bfo)\mid W_o] = s v_s W_o \E \Big[ \one \{W \ge W_o\} W^a \int_{V_o/W^a}^\ff \vp(u) \d u \Big| W_o \Big]. \]
\ene
We further bound the expression in Application~\ref{eq:dego1} using that $\vp$ and $f_W$ are assumed to be regularly varying with suitable indices.
\begin{lemma} \label{lem:regvar}
\[ \E \Big[ \one \{W \ge W_o\} W^a \int_{V_o/W^a}^\ff \vp(u) \d u \Big| W_o \Big] \in O( V_o^{(1-\a)/2}). \]
\end{lemma}
\begin{proof}
We first consider the simple case where $a = 0$.
Then, Karamata's theorem~\cite[Theorem 0.6]{res} implies that for large $V_o$ we have that $\E[\ms{deg}_f(\bfo)\mid W_o] \le c_1v_sW_oV_o \vp(V_o)$, for some constant $c_1>0$.
Again, by the regular variation of $\vp$, for all sufficiently large $V_o$, we have that $\vp(V_o)\le c_2 V_o^{-\a+(\a-1)/2}$.

For $a > 0$, we distinguish between the cases, where $V_o\le M W^a$ and where $V_o> M W^a$, for some large $M>0$.
This allows us to bound the integral with respect to the function $\vp$, and we obtain that 
\[ \E \Big[ W^a \int_{V_o/W^a}^\ff \vp(u) \d u \Big] \le \mu_+((V_o/M)^{1/a}) + c_3 \, V_o \, \E \big[ \one \{ (V_o/M)^{1/a}\ge W \} \, \vp(V_o/W^a) \big], \]
for some constant $c_3 > 0$.
Now, the regular variation of $\vp$ implies that for a suitable $ c_4> 0$ we have 
\[ \int_0^{(V_o/M)^{1/a}} \vp(V_o/w^a) \, \P(\d w) \le c_4 \, V_o^{-\a+\e} \int_0^{(V_o/M)^{1/a}} w^{a\a} \, \P(\d w) \le c_4 \, \mu_{a\a} V_o^{-\a + (\a - 1) / 2}, \]
where we also used that regularly varying functions can be bounded by polynomials with a slightly weaker exponent, see~\cite[Proposition 0.8]{res}.
On the other hand, since $1 - F$ is also regularly varying, for all sufficiently large $V_o$ and for some suitable constants $c_5, c_6 > 0$,
\begin{align*}
\mu_+ ((V_o / M)^{1/a}) &= \int^\ff_{(V_o/M)^{1/a}} w^a \, \P(\d w) \\
&= \f{V_o}{M} (1 - F \big( ({V_o}/{M})^{1/a} \big) + a \int^\ff_{(V_o/M)^{1/a}} r^{a-1} (1-F(r)) \d r \\
&\le c_5 \, V_o^{1 - (\beta - a (\a - 1) / 2) / a},
\end{align*}
where again used~\cite[Theorem 0.6 and Proposition 0.8]{res}.
Now, since $\beta>a\a$ we see that $1-(\beta-a(\a-1)/2)/a<(1-\a)/2$, which finishes the proof.
\end{proof}
After having discussed these specific applications, we now turn to the proof of Lemma~\ref{lem:deg}.
\bep[Proof of Lemma~\ref{lem:deg}]
First, by the Mecke formula~\cite[Theorem 4.1]{LastPenrose}, 
\[ \E[\ms{deg}_f(\bfo)\mid W_o] = s\int_{\R^d} \E \big[ f(|B_{|x|}(o)|,W_o,W_x) p_s(o,W_o; x, W_x) \mid W_o \big] \d x. \]
Now, switching to spherical coordinates, substituting $u := v / (v_s \, \k(W_o, W))$, and applying Fubini's theorem yield
\begin{align*}
	\E[ \ms{deg}_f(\bfo) \mid W_o] &= \int_0^\ff \E \big[f(v, W_o, W) \vp(v / (v_s \k(W_o, W)) \mid W_o\big]\d v \\
	&= v_s \int_{0}^\ff \E \big[ \k(W_o, W) f(u v_s \kappa(W_o, W), W_o, W) \vp(u) \mid W_o\big]\d u. \\
	&= v_s W_o \E \Big[ \one \{ W \ge W_o \} W^a \int_0^\ff f(u v_s W_o W^a, W_o, W) \vp(u) \d u \Big| W_o \Big] \\
	&\phantom=+ v_sW_o^a \E \Big[ \one \{W \le W_o\} W \int_0^\ff f(u v_s W^a_o W, W_o, W) \vp(u) \d u \Big| W_o \Big],
\end{align*}
as asserted.
\enp
As an immediate application of the above discussion, we can prove Lemma~\ref{theorem:Ex_poisson_conv}.
\begin{proof}[Proof of Lemma~\ref{theorem:Ex_poisson_conv}]
By the Mecke formula and reparametrization, we have
\begin{align*}
\E[D_s]&= \E\Big[\sum_{\bfx\in\wt{\PP}_s}\one\{\ms{deg}(\bfx)=k \mbox{ in } G(\wt{\PP}_s, v_s) \}\Big]\\
&= s \P \big( \ms{deg}(\bfo) = k \mbox{ in } G(\wt{\PP}_s, v_s) \big) \\
&= s \E \Big[ \f1{k!}\Big( sv_s h(W_o)\Big)^k e^{-sv_sh(W_o)} \Big| W_o \Big] = 1,
\lab{eq:translation_inv}
\end{align*} 
where we also used that $\ms{deg}(\bfo)$ is a Poisson random variable with parameter
\[ \E [\ms{deg}(\bfo)\mid W_o] = s \int_{\R^d} \E \big[p_s (o, W_o; x, W_x) \mid W_x \big] \d x = s v_s h(W_o). \]
This finishes the proof.
\end{proof}

%
%
\subsection{Proof of Proposition~\ref{pr:neg}}
\label{ss:neg}
\bep[Proof of Proposition~\ref{pr:neg} Part 1: Mark approximation]
We perform the proof in three steps.

{\bf Step 1.} Before performing the main mark approximation, we neglect the largest marks.
For this, let
\[ \xi^c_s= \sum_{\bfx \in \wt{\PP_s}}g^c_k(\bfx,\wt{\PP_s}; v_s)\delta_x, \]
where $g^c_k(\bfx,\wt{\PP_s}; v_s):=g_k(\bfx,\wt{\PP_s}; v_s)\one\{W_x< c\}$ and $c$ is such that $\mu_-(c)>0$.
Then, using that $\xi_s$ and $\xi^c_s$ are defined on the same probability space, Markov's inequality and the Mecke theorem,
\begin{align*}
d_{\rm KR}(\xi_s,\xi^c_s)&\le \E[d_{\rm TV}(\xi_s, \xi_s^c)]=\E[(\xi_s - \xi_s^c)([0,1]^d) ]
\le s\E \big[ \psi^c(s) \mid W_o \big],
\end{align*}
where $\psi^c(s):=\one\{W_o\ge c,\,\deg(o)= k \text{ in }G(\wt{\PP_s},v_s)\}$.
Then, $s \E [\psi^c(s) \mid W_o]= \E[f^c_s (W_o) \one \{W_o\ge c\} \mid W_o ]$, with,
\[ f^c_s(W_o) := \f s{k!} \big( s v_s h(W_o) \big)^k \exp(- s v_s h(W_o)) \le C_ks\exp(-sv_sh(W_o)/2), \]
where we used that $\sup \{ x^k \exp(-x/2) \colon x \ge 0 \} < \ff$.
Now we can further bound,
\begin{align*}
s\exp(-sv_sh(W_o)/2)\one\{W_o\ge c\}&\le s\exp(-sv_s c^a\mu_-(c)/2)=s\exp\big(-sv_sw^\eta_s c_s/2\big),
\end{align*}
with $c_s=c^a\mu_-(c)/w_s^\eta$.
Note that $\lim_{s\tff} c_s \geq 2$ since $w_s \downarrow 0$ and, invoking Assumption~\ref{as:1}, $s v_s w^\eta_s \ge 2 \log s$ for all sufficiently large $s$.
Hence, for $s$ sufficiently large, 
\begin{align*}
s\exp(-sv_s c^a\mu_-(c)/2)\le s\exp(-2\log s)\to0.
\end{align*}

\medskip
{\bf Step 2.} Following the same initial arguments as in Step 1, we now remove marks $W_o\ge w_s^\eta$.
More precisely, we then have that $d_{\rm KR}( \xi^c_s,\xi^{w_s^\eta}_s)\le \E \big[ f_s(W_o) \one \{w_s^\eta \le W_o < c \} \mid W_o \big]$, where
\begin{align*}
f_s(W_o)&:=\f s{k!}\big(sv_sh(W_o)\big)^k\exp(-sv_sh(W_o))
\le C_ks\exp(-sv_sh(W_o)/2).
\end{align*}
Now, we bound slightly differently.
For large values of $s$,
\[ s \, \exp(-s v_s h(W_o) / 2) \one\{ w_s^\eta \le W_o < c \} \le s \, \exp(-s v_s w^\eta_s \mu_+(c)). \]
Again, using Assumption~\ref{as:1}, the right-hand side tends to zero as $s\tff$.

\medskip
{\bf Step 3:} We now come to the main mark-approximation step.
We may bound, as above, 
\[ d_{\rm KR}(\xi^{w_s^\eta}_s, \hat\xi^{w_s^\eta}_s) \le s \, \E \big[ \psi(s) \mid W_o \big], \]
where 
$\hat\xi^{w_s^\eta}_s= \sum_{x \in \wt{\PP}_s}\hat g^{w_s^\eta}_k(x,\wt{\PP_s}; v_s) \de_x$ and $\psi(s)=\psi_1(s)\one\{W_o< w_s^\eta\}+\psi_2(s)\one\{W_o< w_s^\eta\}$ with
\begin{align*}
\psi_1(s)&:=\one\{\deg(o)=k \text{ in }G(\wt{\PP}_s^{\ge W_o},v_s)\}\one\{\deg(o)>0 \text{ in }G(\wt{\PP}_s^{< W_o},v_s)\}\\
\psi_2(s)&:=\one\{\deg(o)=k \text{ in }G(\wt{\PP}_s,v_s)\}\one\{\deg(o)>0 \text{ in }G(\wt{\PP}_s^{< W_o},v_s)\}.
\end{align*}
Here, $\wt{\PP}_s^{\ge W_o}$ denotes the Poisson point process $\wt{\PP}_s$ restricted to points with marks $\ge W_o$.
In words, $\psi(s)$ bounds the indicators of the two events that $\hat\xi^{w_s^\eta}_s$ contains a point not contained in $\xi^{w_s^\eta}_s$ and vice versa.

Using the fact that $\wt{\PP}_s$ is an independent superposition of $\wt{\PP}_s^{\ge W_o}$ and $\wt{\PP}_s^{< W_o}$, we can use the Mecke formula to write $s \E \big[ \psi_1(s) \one \{W_o< w_s^\eta\} \mid W_o \big] = \E \big[ f^1_s(W_o) \one \{W_o< w_s^\eta \} \mid W_o \big]$, with
\[ f^1_s(W_o) := \f s{k!} \big( s v_s W_o \mu_+(W_o) \big)^k \exp(- s v_s W_o \mu_+(W_o)) \big(1 - \exp(-s v_s W^a_o \mu_-(W_o)) \big). \]
Now, under the event $\{W_o< w_s^\eta\}$, we have that $sv_s W^a_o \mu_-(W_o) \le s v_s w_s^{(a+1)\eta} F(w_s^\eta) \in o(1)$, by Assumption~\ref{as:2}.
Indeed, invoking Lemma~\ref{lem:estimates}, it suffices to argue that $\log(s) w_s^{(a+1) \eta - 1} F(w_s^\eta) \in o(1)$.
But this is the case since $w_s^{(a+1) \eta - 1} \le w_s^{- (1 - \eta)}$.
Hence,
\[ \E [f^1_s(W_o) | W_o] \le \exp(s v_s w_s^{(a+1) \eta} F(w_s^\eta)) \big( 1 - \exp(-s v_s w_s^{(a+1) \eta} \, F(w_s^\eta)) \big) \to 0. \]

For the second term $\psi_2$, we write $s \E \big[ \psi_2(s) \one\{W_o< w_s^\eta\} \mid W_o \big] = \E \big[f^2_s(W_o) \one \{ W_o < w_s^\eta\} \mid W_o \big]$, with
\begin{align*}
f^2_s(W_o) &:= s \sum_{l=1}^k \f1{l!} \big( s v_s W^a_o \mu_-(W_o) \big)^l \exp(-s v_s W^a_o \mu_-(W_o)) \\
&\qquad \f1{(k-l)!} (s v_s W_o \mu_+(W_o))^{k-l} \exp(- s v_s W_o \mu_+(W_o))\\
&=\f s{k!} \big( s v_s h(W_o) \big)^k \exp(- s v_s h(W_o)) \sum_{l=1}^k \binom{k}{l} \left( \f{W^a_o \mu_-(W_o)}{h(W_o)} \right)^l \left( \f{W_o \mu_+(W_o)}{h(W_o)}\right)^{k-l}.
\end{align*}
Again, invoking~\eqref{eq:scaling}, it suffices to show that 
\begin{align*}
\sum_{l=1}^k\binom{k}{l}&\left(\f{W^a_o\mu_-(W_o)}{h(W_o)}\right)^l\left(\f{W_o\mu_+(W_o)}{h(W_o)}\right)^{k-l}= 1-\left(\f{W_o\mu_+(W_o)}{h(W_o)}\right)^{k}
\end{align*}
tends to zero by Part 4 of Lemma~\ref{lem:estimates}.
\enp

\bep[Proof of Proposition~\ref{pr:neg} Part 2: Reach approximation]
We again invoke the Markov inequality and the Slivnyak--Mecke formula to estimate,
\begin{align*}
d_{\rm KR}(\hat\xi^{w_s^\eta}_s, \bar\xi^{V_o(s), w_s^\eta}_s) \le s \E \big[ \psi'(s) \mid W_o \big],
\end{align*}
where $\psi' = \psi'_1 \one \{ W_o < w_s^\eta \} + \psi'_2 \one \{W_o < w_s^\eta\}$ with $E_s := B_{(v_s W_o V_o(s) / \nu_d)^{1/d}}(o)$ and 
\begin{align*}
\psi'_1(s)&:=\one\{\deg(o)=k \text{ in }G(\wt{\PP}^{\ge W_o}_s\cap E_{s},v_s)\}\one\{\deg(o)> 0 \text{ in }G(\wt{\PP}^{\ge W_o}_s\cap E^c_{s},v_s)\} \\
\psi'_2(s)&:=\one\{\deg(o)=k \text{ in }G(\wt{\PP}^{\ge W_o}_s,v_s)\}\one\{\deg(o)> 0 \text{ in }G(\wt{\PP}^{\ge W_o}_s \cap E^c_{s},v_s)\}
\end{align*}
Recall that we set $V_o(s) = w_s^{K(\eta-1)}$ for $\eta \in (0,1)$ and $K > 0$ and hence $\lim_{s \tff} V_o(s) = \infty$.
Note that, conditioned on $W_o$, the Poisson point processes $\wt{\PP}^{\ge W_o}_s\cap E_{s}$ and $\wt{\PP}^{\ge W_o}_s\cap E^c_{s}$ are independent and thus, $s \E \big[ \psi'_1(s) \mid W_o \big] = \E \big[ f'^1_s(W_o) \one\{ W_o < w_s^\eta \} \mid W_o \big]$ with
\begin{align*}
f'^1_s(W_o) &:= \f s{k!} \big( s v_s W_o \bar\mu_+(W_o, V_o(s)) \big)^k \, \exp(-s v_s W_o \bar\mu_+(W_o, V_o(s))) \times \\
&\qquad \big( 1 - \exp(-s v_s W_o \bar\mu_-(W_o, V_o(s)) \big) 
\end{align*}
where we used the notation 
\begin{align*}
\bar\mu_+(W_o, V_o) &:= \E \Big[ \one \{ W \ge W_o \} W^a \int \one\{u \le V_o / W^a\} \vp(u) \d u \mid W_o \Big] \text{ and } \\
\bar\mu_-(W_o, V_o) &:= \E \Big[ \one \{W \ge W_o\} W^a \int \one \{u \ge V_o/W^a\} \vp(u) \d u \mid W_o \Big].
\end{align*}
As before, we use the first part of $f'^1_s$ to compensate for the coefficient $s$ and the second part to achieve the convergence to zero.
With $\de = (\a - 1) / 2$ and noting that $W_o \bar \mu_+(W_o, V_o(s)) \le h(W_o)$ we have
\begin{align*}
f'^1_s(W_o)\le \f s{k!} \big(s v_s h(W_o) \big)^k \, \exp(-s v_s h(W_o)) &\exp \Big( s v_s w_s^\eta\big(c w_s^{K(1 - \eta) \de} + w_s^{a\eta} F(w_s^\eta) \big) \Big) \\
&\times\Big(1-\exp\big(-csv_sw_s^{\eta + K(1 - \eta)\de}\big)\Big),
\end{align*}
where we used that, employing Lemmas~\ref{lem:estimates} and~\ref{lem:regvar},
\[ h(W_o) - W_o \bar\mu_+(W_o, V_o) = W_o \bar\mu_-(W_o, V_o) + W_o^a \mu_-(W_o) \le c W_o V_o^{-\de} + W_o^{a+1} F(W_o). \]
Arguing as in the third step of the mark-approximation proof above, by Assumptions~\ref{as:2}, we have that $\exp(s v_s w_s^{(a+1)\eta} F(w_s^\eta))) \to 1$, and hence, using~\eqref{eq:scaling}, 
\[ \E[f'^1_s(W_o) \mid W_o] \le C \exp \big( - c s v_s w_s^{\eta + K(1 - \eta)\de} \big) \Big( 1 - \exp \big(- c s v_s w_s^{\eta + K(1 - \eta) \de} \big) \Big). \]
In order to see that $sv_sw_s^{\eta + K(1 - \eta)\de} \in o(1)$, we invoke again Part 3 of Lemma~\ref{lem:estimates} and Assumption~\ref{as:3} to see that for some $c' > 0$,
\[ s v_s w_s^{\eta + K(1 - \eta) \de}\le c' \log(s) w_s^{\eta + K(1 - \eta) \de-1} \to 0. \]

\medskip
For $\psi'_2$, we can argue similarly to the second error term in Step 3 of the proof of the mark approximation.
More precisely, we have $s\E \big[ \psi'_2(s) \one\{W_o < w_s^\eta\} \mid W_o \big] = \E \big[ f'^2_s(W_o) \one \{W_o<w_s^\eta\} \mid W_o \big]$ with $f'^2_s(W_o)$ defined as
\begin{align*}
	&s\sum_{l=1}^k\f1{l!}\big(sv_sW_o\bar\mu_-(W_o,V_o(s))\big)^l\exp(-sv_sW_o\bar\mu_-(W_o,V_o(s)))\times\\
	&\f1{(k-l)!}\big(sv_sW_o\bar\mu_+(W_o,V_o(s))\big)^{k-l}\big(1-\exp(-sv_sW_o\bar\mu_+(W_o,V_o(s))\big)\\
	&=\f{s}{k!}\big(sv_sW_o\mu_+(W_o)\big)^ke^{-sv_sW_o\mu_+(W_o)}\sum_{l=1}^k\binom{k}{l}\left(\f{\bar\mu_-(W_o,V_o(s))}{\mu_+(W_o)}\right)^l
\left(\f{\bar\mu_+(W_o,V_o(s))}{\mu_+(W_o)}\right)^{k-l}\\
	&\le \f{C s}{k!}\big(sv_sh(W_o)\big)^ke^{-sv_sh(W_o)}e^{sv_sw^{(a+1)\eta}_oF(w_s^{\eta})}\Big[1-\left(\f{\bar\mu_+(W_o,V_o(s))}{\mu_+(W_o)}\right)^{k}\Big].
\end{align*}
Again, using~\eqref{eq:scaling} and Assumption~\ref{as:2}, it suffices to show that, under the event $W_o< w_s^\eta$,
\[ \f{\bar\mu_-(W_o,V_o(s))}{\bar\mu_+(W_o,V_o(s))}\le \f{\E\Big[W^a\int^\ff_{ V_o(s)/W^a}\varphi(u)\d u\Big]}{\E\Big[W^a\one\{W\ge w_s^\eta\}\int^\ff_{ V_o(s)/W^a}\varphi(u)\d u\Big]}\in o(1). \]
This is true since, first for the numerator, for all $u$, we have $\E [W^a \one\{W^a\ge V_o(s)/u\} ] \le \mu_a$ and $\lim_{s\tff} \E [ W^a \one \{W^a \ge V_o(s)/u \} ] = 0$ and thus, using dominated convergence,
\begin{align*}
	\lim_{s\tff}\E\Big[W^a\int^\ff_{V_o(s)/W^a}\varphi(u)\d u\Big]&=
	\int_0^\ff \varphi(u)\E\Big[W^a\one\{W^a\ge V_o(s)/u\}\Big]\d u=0.
\end{align*}
On the other hand, for the denominator, for all $u$, we have $\E\Big[W^a\one\{W\ge w_s^\eta\}\one\{W^a\le V_o(s)/u\}\Big]\le \mu_a$ and $\lim_{s\tff}\E\Big[W^a\one\{W\ge w_s^\eta\}\one\{W^a\le V_o(s)/u\}\Big]=\mu_a$ and thus, again by dominated convergence, 
\begin{align*}
	\lim_{s\tff}\E\Big[W^a\int^\ff_{ V_o(s)/W^a}\varphi(u)\d u\Big]&=
\int_0^\ff  \varphi(u)\E\Big[W^a\one\{W^a\ge V_o(s)/u\}\Big]\d u=\mu_a.
\end{align*}
This finishes the proof.
\enp

%
%
\subsection{Proof of Proposition~\ref{pr:qa}}
\label{sec:qa}
To prove Proposition~\ref{pr:qa}, we employ~\cite[Theorem 4.1]{BSY202}.
To express the considered functional in the framework of~\cite[Theorem 4.1]{BSY202}, we first introduce additional notation.
To each point $\bfx$ we associate a deterministic compact set $S^V(\bfx)$ from which the score function of interest can be computed with high probability.
In addition to $S^V(\bfx)$,~\cite[Theorem 4.1]{BSY202} also allows for the use of a more refined localization set $\cS(\bfx,\om)$, which may be random in general.
In the present setting, we do not need this additional flexibility since, at the beginning of Section~\ref{sec:proofs}, we have already implemented a truncation step.
Therefore, we set $\cS(\bfx,\om):=S^V(\bfx):=B_{(v_sW_xV/\nu_d)^{1/d}}(x)$, to be the ball of volume $v_sW_xV$ around $x$.

Then,~\cite[Theorem 4.1]{BSY202} bounds the KR-distance between the process of interest and a Poisson point process by a sum of four quantities.
The first of them is the total variation between the corresponding intensity measures.
The remaining three quantities, denoted by $E_1, E_2, E_3$ concern higher-order deviations.
Note that, since we choose $\cS(\bfx,\om)=S^V(\bfx)$, the $E_1$ term is identically 0.
Hence, we formally state three remaining separate auxiliary results, Lemmas~\ref{lem:par}--\ref{lem:e3} below.
The proofs follow afterwards.

We begin with the intensity measures.
By the homogeneity of the approximations, the intensity measure $L_s(\d x)$ of $\bar \xi_s^{V_o(s),w_s^\eta}$ has the constant Lebesgue density
$\int_\Om\E[\bar g_k^{ V_o(s),w_s^\eta}(\bfx,\wt{\PP}_s)]\bK_s(\d\bfx)$, where $\bar g_k^{ V_o(s),w_s^\eta} \equiv \bar g_k^{v_s W_x V_o(s),w_s^\eta}$ for brevity.
\bel[Convergence of intensity measures]	\label{lem:par}
We have that
\beqn\label{eq:E_4}
\begin{split}
\lim_{s\tff}d_{\rm TV}(L_s,{\rm{Leb}})=0.
\end{split}
\eeqn
\enl

As described above, the following statement is immediate.
\bel[Convergence of $E_1$] \label{lem:e1}
We have that
\beqn
\lim_{s\tff}\int_{\Om} \E\left[ \one\big\{\cS(\bfx,\wt{\PP}_s)\not\su S^{V_o(s)}(\bfx)\big\} \bar g_k^{V_o(s),w_s^\eta}(\bfx,\wt{\PP}_s)\right]\bK_s(\d\bfx)=0.
\label{eq:E_1}
\eeqn
\enl
%
%
Here are the remaining requirements.
\bel[Convergence of $E_2$] \label{lem:e2}
We have that
\beqn\label{eq:E_2}
\begin{split}
\lim_{s\tff} \int_{\Om^2} \one \{S^{V_o(s)}(\bfx) \cap S^{V_o(s)}(\bfz)\ne \es\} &\E [\bar g_k^{V_o(s), w_s^\eta}(\bfx, \wt{\PP}_s)] \E[\bar g_k^{V_o(s),w_s^\eta}(\bfz,\wt{\PP}_s)] \bK_s(\d\bfz)\bK_s(\d\bfx)=0.
\end{split}
\eeqn
\enl
%
%
\bel[Convergence of $E_3$]
	\label{lem:e3}
We have that
\beqn\label{eq:E_3}
\begin{split}
\lim_{s\tff} \int_{\Om^2} &\one\{S^{V_o(s)}(\bfx) \cap S^{V_o(s)}(\bfz)\ne \es\} \times \\
&\E[\bar g_k^{V_o(s),w_s^\eta}(\bfx,\wt{\PP}_s \cup \{\bfz\})\bar g_k^{V_o(s),w_s^\eta}(\bfz,\wt{\PP}_s\cup\{\bfx\})]\bK_s(\d\bfz)\bK_s(\d\bfx) = 0.
\end{split}
\eeqn
\enl

%
%
\bep[Proof of Lemma~\ref{lem:par}]
First note that 
\begin{align*}
d_{\rm TV}(L_s,{\rm{Leb}}) &\le \Big| \int_\Om\E[\bar g_k^{V_o(s),w_s^\eta}(\bfx,\wt{\PP}_s)]\bK_s(\d\bfx) - 1 \Big| \\
&\le \int_\Om \E[|\bar g_k^{V_o(s),w_s^\eta}(\bfx,\wt{\PP}_s) - g_k(\bfx,\wt{\PP}_s)|] \bK_s(\d\bfx) \\
&\le s \, \E [|\bar g_k^{V_o(s),w_s^\eta}(\bfo,\wt{\PP}_s) - g_k(\bfo,\wt{\PP}_s)| \mid W_o],
\end{align*}
where we used Lemma~\ref{theorem:Ex_poisson_conv} and the Mecke formula.
Step-by-step reintroducing the mark- and reach approximations, we see that 
\begin{align*}
s \, \E^o[|\bar g_k^{V_o(s),w_s^\eta}(\bfo,\wt{\PP}_s)-g_k(\bfo,\wt{\PP}_s)|] &\le s\E^o[\psi^c(s)+\psi(s)+\psi'(s)],
\end{align*}
where the right-hand-side tends to zero as $s$ tends to infinity using the exact same arguments as in the proof of Proposition~\ref{pr:neg}.
\enp

%
%
\bep[Proof of Lemma~\ref{lem:e2}]
By symmetry, we can insert $2 \one\{W_z \ge W_x\}$.
Then, under this event, $\one\{S^{V_o(s)}(\bfx)\cap S^{V_o(s)}(\bfz)\ne \es\}\le \one\{x\in B_{2(v_sV_o(s)W_z/\nu_d)^{1/d}}(z)\}$
and hence, also using that $W_z\le w_s^\eta$ and translation invariance, the integral on the left-hand side of~\eqref{eq:E_2} is bounded from above by
\begin{equation*}
\begin{split}
2^dv_sw_s^\eta V_o(s)&\E[s\bar g_k^{V_o(s),w_s^\eta}(\bfo,\wt{\PP}_s)]^2,
\end{split}
\end{equation*}
where $\lim_{s\uparrow\infty}\E[s\bar g_k^{V_o(s),w_s^\eta}(\bfo,\wt{\PP}_s)]=1$ by Lemma~\ref{lem:par}.
Hence, using $1\le2sF(w_s^\eta)$ and Part 3 of Lemma~\ref{lem:estimates}  we have for some $c>0$, 
\begin{equation*}
\begin{split}
	v_sw_s^{\eta-K(1 - \eta)}\le c\log(s) w_s^{-(K + 1)( 1 -\eta)} F(w_s^\eta),
\end{split}
\end{equation*}
which by Assumption~\ref{as:2} tends to zero as $s\tff$.
\enp

%
%
\bep[Proof of Lemma~\ref{lem:e3}]
Again, by symmetry, we may insert the indicator of the event $\one\{W_z \le W_x\}$, to obtain $\one\{S^{V_o(s)}(\bfx)\cap S^{V_o(s)}(\bfz)\ne \es\}\le \one\{x\in B_{2(v_sV_o(s)W_x/\nu_d)^{1/d}}(z)\}$.
The important observation is that $\bar g_k^{V_o(s),w_s^\eta}(\bfx,\wt{\PP}_s\cup\{\bfz\})$ only takes into account nodes with weight exceeding $W_x$ and therefore the point $\bfz$ can be neglected.
As $W_z \le w_s^\eta$, the integral in the left-hand side of~\eqref{eq:E_3} is bounded above by
\begin{align*}
2F(w_s^\eta)\int_{\Om^2} &\one\{S^{V_o(s)}(\bfx)\cap S^{V_o(s)}(\bfz)\ne \es\} \E[\bar g_k^{V_o(s),w_s^\eta}(\bfx,\wt{\PP}_s)]\bK_s(\d \bfz)\bK_s(\d \bfx)\\
&\le 2F(w_s^\eta)\int_{\Om^2} \one\{x\in B_{2(v_sV_o(s)w_s^\eta)^{1/d}}(z)\} \E[\bar g_k^{V_o(s),w_s^\eta}(\bfx,\wt{\PP}_s)]\bK_s(\d \bfz)\bK_s(\d \bfx)\\
&\le 4sv_sw_s^\eta V_o(s) F(w_s^\eta)\E[s\bar g_k^{V_o(s),w_s^\eta}(\bfo,\wt{\PP}_s)],
\end{align*}
where $\lim_{s\uparrow\infty}\E[s\bar g_k^{V_o(s),w_s^\eta}(\bfo,\wt{\PP}_s)]=1$ by Lemma~\ref{lem:par}.
For this, invoking again Part 3 of Lemma~\ref{lem:estimates}, we have for some $c > 0$ 
\[ \begin{split} s v_s w_s^{\eta - K (1 - \eta)} F(w_s^\eta) \le c \log(s) w_s^{-(K + 1) (1 - \eta)} F(w_s^\eta) \end{split} \]
which by Assumption~\ref{as:2} tends to 0 as $s\tff$.
\enp

\paragraph{\bf Acknowledgement.} BJ and SKJ received support by the Leibniz Association within the Leibniz Junior Research Group on \textit{Probabilistic Methods for Dynamic Communication Networks} as part of the Leibniz Competition (grant no.\ J105/2020).
This work was supported by the Danish Data Science Academy, which is funded by the Novo Nordisk Foundation (NNF21SA0069429) and Villum Fonden (40516).
\bibliographystyle{abbrv}
\bibliography{literature}

\end{document}